\documentclass[graybox]{svmult}

\usepackage{type1cm}        
\usepackage{makeidx}         
\usepackage{graphicx}        
\usepackage{multicol}        
\usepackage[bottom]{footmisc}

\usepackage{newtxtext}       %
\usepackage{newtxmath}       

\newcommand{\field}[1]{\mathbb{#1}}
\newcommand{\bR}{\field{R}}
\newcommand{\bN}{\field{N}}

\def\rd{\bR^d}
\def\rd{\mathbb{R}^d}
\def\f{\varphi}

\def\cS{\mathcal{S}}

\def\rd{\bR^d}

\def\rdd{{\bR^{2d}}}

\def\lrd{L^2(\rd)}

\def\<{\left<}
\def\>{\right>}

\def\mv1{M_v^1}

\def\phas{(x,\omega )}

\def\E{\mathcal{E}}
\def\Fur{\mathcal{F}}

\newcommand{\tfs}{time-frequency shift}
\makeindex             

\begin{document}

\title*{Signal Analysis using Born-Jordan-type Distributions}
\author{Elena Cordero, Maurice de Gosson, Monika D\"{o}rfler and Fabio Nicola}
\institute{Elena Cordero \at Universit\`a di Torino, Dipartimento di
	Matematica, via Carlo Alberto 10, 10123 Torino, Italy \email{elena.cordero@unito.it} \and Maurice de Gosson \at University of Vienna, Faculty of Mathematics,
	Oskar-Morgenstern-Platz 1 A-1090 Wien, Austria
\email{maurice.de.gosson@univie.ac.at} \and Monika D\"{o}rfler \at University of Vienna, Faculty of Mathematics,
Oskar-Morgenstern-Platz 1 A-1090 Wien, Austria \email{monika.doerfler@univie.ac.at} \and Fabio Nicola \at Dipartimento di Scienze Matematiche, Politecnico di Torino, corso
Duca degli Abruzzi 24, 10129 Torino, Italy \email{fabio.nicola@polito.it}}

%
%
\maketitle

\abstract{In this note we exhibit recent advances in signal analysis via time-frequency distributions.
New members of the Cohen class, generalizing the Wigner distribution, reveal to be effective in damping artefacts of some signals. We will survey their main properties and drawbacks and  present open problems related to such phenomena.}

\section{Introduction}
\label{Int:1}
The Wigner distribution (Wigner transform, Wigner function or Wigner-Ville distribution) has a long tradition which started as a probability quasi-distribution in 1932 with  Eugene Wigner's ground-breaking paper  \cite{Wigner1932}.
In 1948 it was reinvented  by  Jean Ville in \cite{Ville1948} as a quadratic representation of the local time-frequency energy of a signal. This was the starting point for its numerous applications in signal analysis: from electrical engineering and communication theory to any field involving the problem of treating signals: seismology, biology, medicine etc.

Given two functions  $f,g\in L^2(\mathbb{R}^{d})$,
their (cross-)Wigner distribution is defined to be
	\begin{equation}\label{WD}
	W(f,g)\phas =\int_{\mathbb{R}^{d}}e^{-2\pi i
		\omega y}f(x+\tfrac{1}{2}y)\overline{g(x-\tfrac{1}{2}y)}dy,
	\end{equation}
	($\omega y= \omega\cdot y$ denoting the scalar product in $\mathbb{R}^{d}$).
If 	$f=g$ we set $Wf:= W(f,f)$, named the Wigner distribution of $f$.
	The quadratic nature of the Wigner distribution $Wf$ causes the appearance of interferences between the distinct components of the signal. Roughly speaking, if a signal $f$ is sum of two components $f_1,f_2$, then its Wigner distribution becomes
\begin{equation*}  
W(f_1+f_2)=Wf_1+Wf_2+2{\mathcal Re}{W(f_1,f_2)}.
\end{equation*}
The cross-term $W(f_1,f_2)$ appearing above produces unwanted interferences. This will be made clear in the following examples.

Recall the translation and modulation operators:
\[
T_x f(y)=f(y-x),\quad M_\omega f(y)=e^{2\pi i y\omega}f(y),\quad x,\omega\in \rd,
\]
which combined  are called time-frequency shifts:
\[
\pi(z) f(y)=M_\omega T_x f(y)= e^{2\pi i y\omega}f(y-x),\quad z=(x,\omega).
\]
We can show such unwanted phenomenum for a signal $f$ that is the sum of 
four Gaussian atoms,  that is, \tfs s of the Gaussian function. For example,  in dimension $d=1$, consider the  Gaussian function $\f(t)=e^{-\pi t^2}$ and the signal  
 $f$ that is the  sum of the following  $4$ time-frequency shifts of $\f$:
\begin{equation}\label{signalf}
f=\pi(20, 0.25)\f+\pi(40, 0.15)\f+ \pi(40,0.35)\f+\pi(60,0.25)\f.
\end{equation}
\begin{center}
\includegraphics[width=0.9\textwidth]{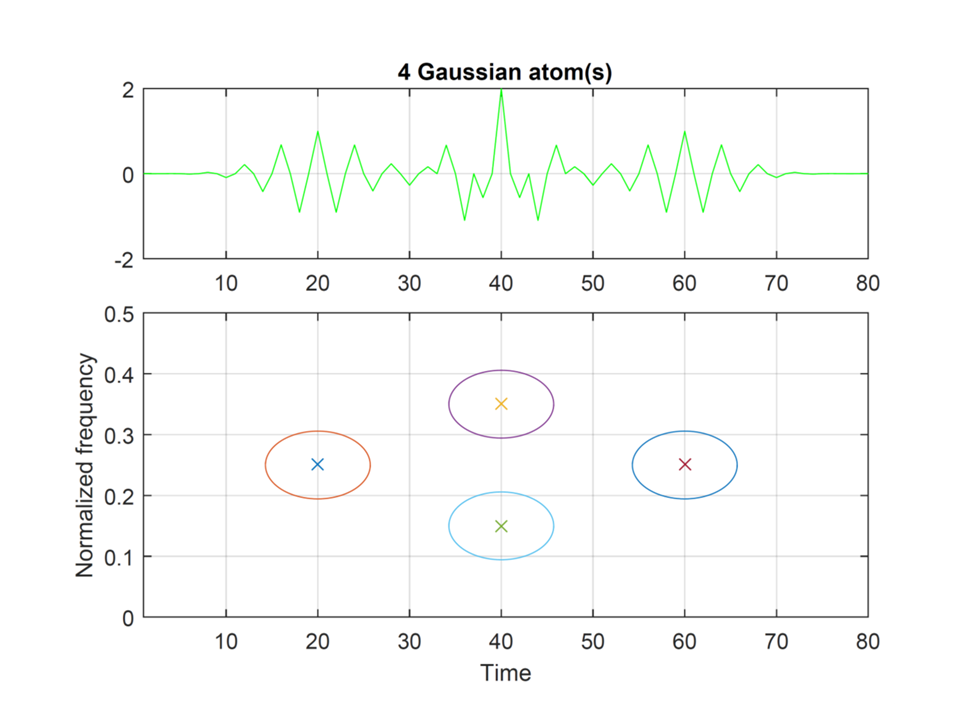}
\end{center}
The Wigner Distribution $Wf$ is represented by the picture below:
\begin{center}
\includegraphics[width=1.1\textwidth]{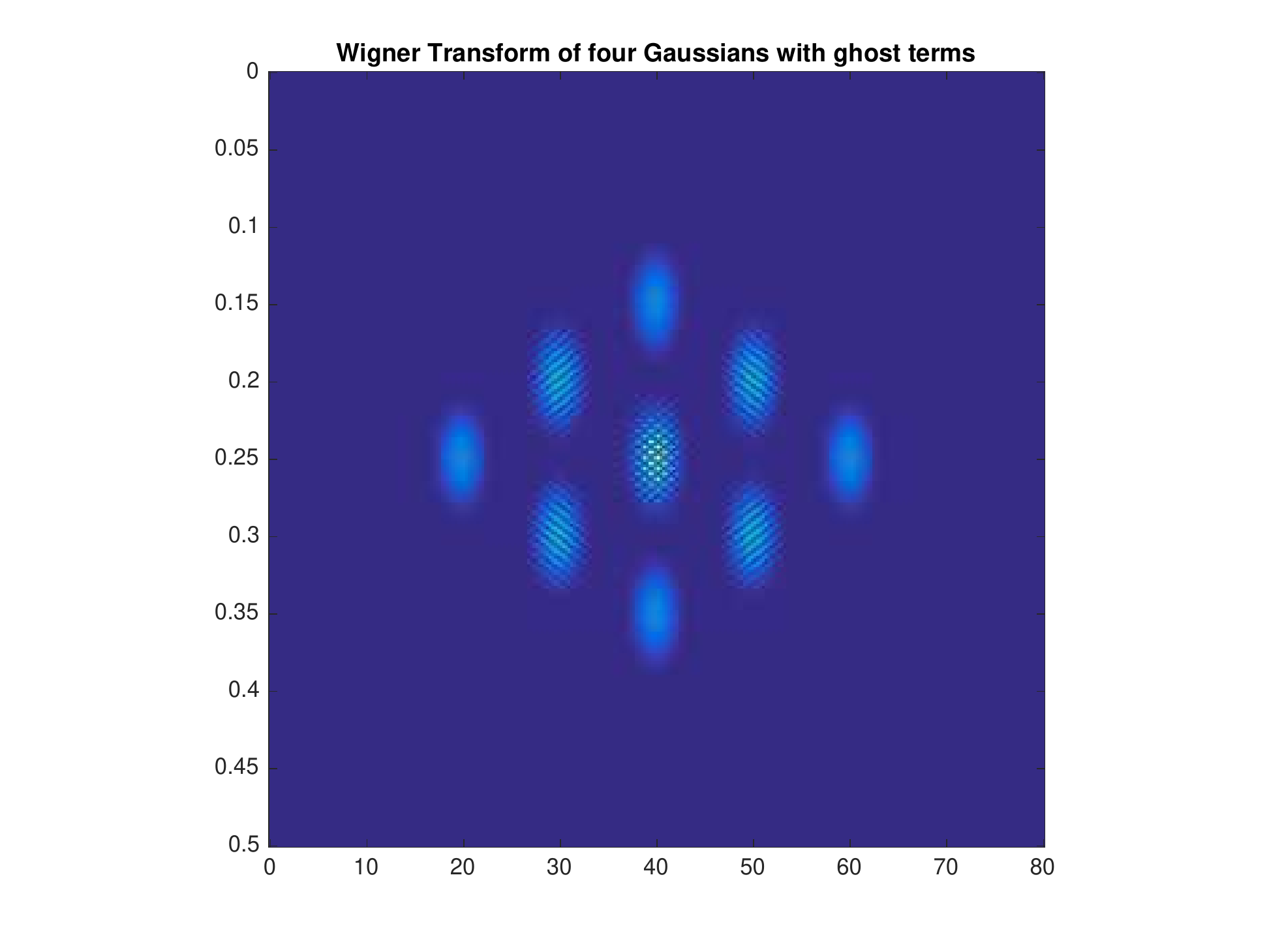}
\end{center}
\par

The four spots placed at the vertices of the rhombus are the real time-frequency content of the signal, whereas the other five spots  represent the  ``ghost frequencies'', as defined  in \cite{bogetal}, arising from the interferences of the four  components of the signal. 

To overcome this undesirable phenomenon, {\it reduced interference distributions} were proposed by Leon Cohen in  \cite{Cohen1}, see also the textbooks \cite{Cohen2,auger}.

The idea underneath the introduction of the so-called Cohen class is to reduce the interferences by a smoothing operation obtained using a convolution product.  Precisely,  we define the Cohen class a follows. 
\begin{definition}\label{Cohenclass}
A member of the Cohen class $Qf$ is a quadratic time-frequency representation  obtained by convolving the Wigner
function $Wf$ with a distribution $\theta\in\mathcal{S}^{\prime
}(\mathbb{R}^{2d})$ (called  Cohen kernel), that is
\begin{equation}
\label{Cohenkernel}Qf=Wf\ast\theta,\quad f\in \cS(\rd).
\end{equation}
\end{definition}
For $f\in\cS(\rd)$ it is easy to check that $Wf\in\cS(\rdd)$ (see, e.g., \cite{deGossonWigner}) so that  $Qf$ is a well-defined tempered distribution for every $f\in\cS(\rd)$. \par 
A possible choice for the Cohen kernel  is $\theta=\mathcal{F}_{\sigma}\Theta^{1}$, with
$\mathcal{F}_{\sigma}\Theta^{1}$ being the symplectic Fourier transform of the
function
\begin{equation}
\label{sincxp}\Theta^{1}(x,\omega)=\mathrm{sinc}(x\omega)=%
\begin{cases}
\displaystyle\frac{\sin(\pi x\omega)}{\pi x\omega} & \mbox{for}\, x\omega
\neq0\\
1 & \mbox{for}\, x\omega=0.
\end{cases}
\end{equation}
 In
this way we obtain the Born-Jordan (BJ) distribution:
\begin{equation}
\label{bj}Q^{1} f= Wf \ast\mathcal{F}_{\sigma}(\Theta^{1}),\quad
f\in L^{2}(\mathbb{R}^{d}),
\end{equation}
see \cite{bogetal, Cohen1,Cohen2,Cohenbook,cgn0,auger,SpringerMG} and the
references therein.

Such distribution was first introduced in 1925 as a quantization rule by the phisicists M. Born and P. Jordan  \cite{bj} and later widely  employed by engineers for its \textit{smoothing effects}, cf. the textbook \cite{auger}.

Examples for tests and real-world signals show in the BJ distribution:
\begin{itemize}
	\item[a)]   ``ghost frequencies'' (arising from the interferences of the several components which do not share the same time or frequency localization) are damped very well, see \cite{bogetal,bo3}.
	\item[b)] The interferences arranged along the horizontal and vertical direction are substantially kept.
	\item[c)] The noise is, on the whole, reduced.
\end{itemize}
The following picture represents the BJ distribution of the signal $f$ in \eqref{signalf}.
\begin{center}
\includegraphics[width=1.2\textwidth]{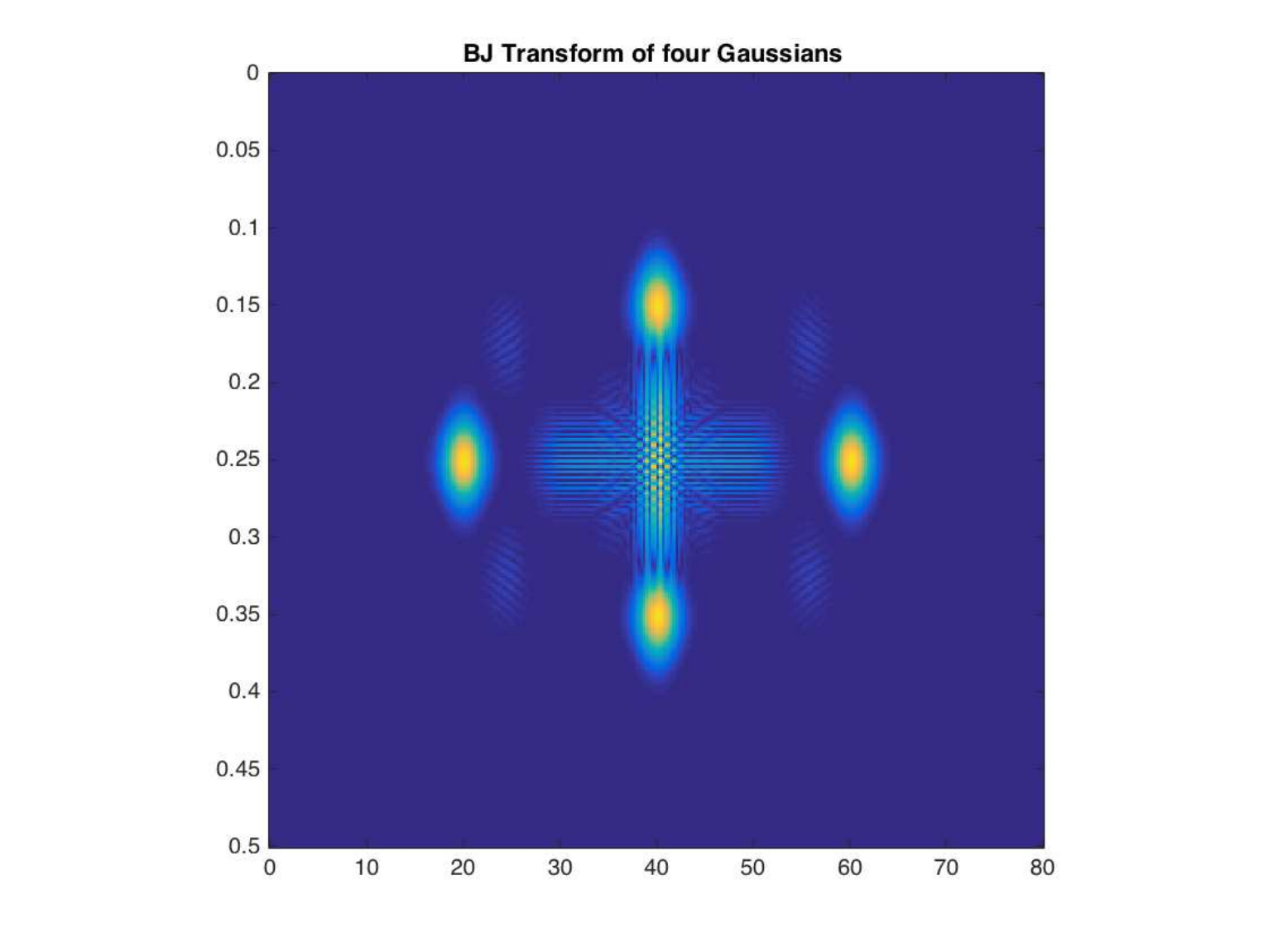}
\end{center}
The smoothing effect of the BJ distribution can be seen quite clearly. But still there are artefacts that one would like to damp. This is the main purpose of the contribution \cite{BJsplines2019}, where new distributions from the Cohen class were proposed and studied.  
Here we present the most relevant issues of the contribution above, highlighting the mathematical explanation of the smoothing effects of the distributions.

Following the suggestions of Jean-Pierre Gazeau, in the manuscript \cite{BJsplines2019}  new interesting Cohen kernels were proposed using the
B-spline functions $B_{n}$.\par The sequence of B-splines
$\{B_{n}\}_{n\in\mathbb{N}_{+}}$ is defined inductively as follows. The first
B-Spline is
\[
B_{1}(t)=\chi_{\left[ -\frac12,\frac12\right] }(t),
\]
and the spline $B_{n+1}$ is 
\begin{equation}
\label{bsplines}B_{n+1}(t)=(B_{n}\ast B_{1})(t)=\int_{\mathbb{R}}%
B_{n}(t-y)B_{1}(y)dy=\int_{-\frac12}^{\frac12} B_{n}(t-y)dy.
\end{equation}
The spline $B_{n}$ is a piecewise polynomial of degree at most $n-1$,
$n\in\mathbb{N}_{+}$, and satisfying $B_{n}\in\mathcal{C}^{n-2}(\mathbb{R})$,
$n\geq2$. For the main properties we refer, e.g., to \cite{Christensen2016}.

The $sinc$ function is recaptured as $\mathrm{sinc}(\xi)=\mathcal{F} B_{1}(\xi)$ and by induction we
infer
\begin{equation}
\label{Fouriersplines}\mathrm{sinc}^{n}(\xi)=\mathcal{F} B_{n}(\xi),\quad
n\in\mathbb{N}_{+}.
\end{equation}
This suggests the following definition.
\begin{definition}
	For $n\in\mathbb{N}$, the nth Born-Jordan kernel  on ${\mathbb{R}^{2d}}$ is
	defined by
	\begin{equation}
	\label{nCohenkerneln}\Theta^{n}(x,\omega)=\mathrm{sinc}^{n}(x\omega
	),\quad(x,\omega)\in{\mathbb{R}^{2d}}.
	\end{equation}
	The related  Born-Jordan distribution of order $n$ (BJDn) is
	\begin{equation}
	\label{e17}Q^{n} f=Wf\ast\mathcal{F}_{\sigma}(\Theta^{n}),\quad f\in L^{2}%
	(\mathbb{R}^{d}).
	\end{equation}
	The cross-BJDn is given by
	\begin{equation}
	\label{crossBJ}Q^{n} (f,g)=W(f,g)\ast\mathcal{F}_{\sigma}(\Theta^{n}),\quad
	f,g\in L^{2}(\mathbb{R}^{d}).
	\end{equation}
	We write $Q^{n} (f,f)=Q^{n} f$, for every $f\in L^{2}(\mathbb{R}^{d})$.
	
	For $n=0$, $\Theta^{0}\equiv1$ and $\mathcal{F}_{\sigma}(1)=\delta$, so that
	$Q^{0}f=Wf$, the Wigner distribution of the signal $f$.
\end{definition}
In the picture below we computed the Born-Jordan distribution of order $3$ of the signal in \eqref{signalf}.
\begin{figure}
	\includegraphics[width=1.1\textwidth]{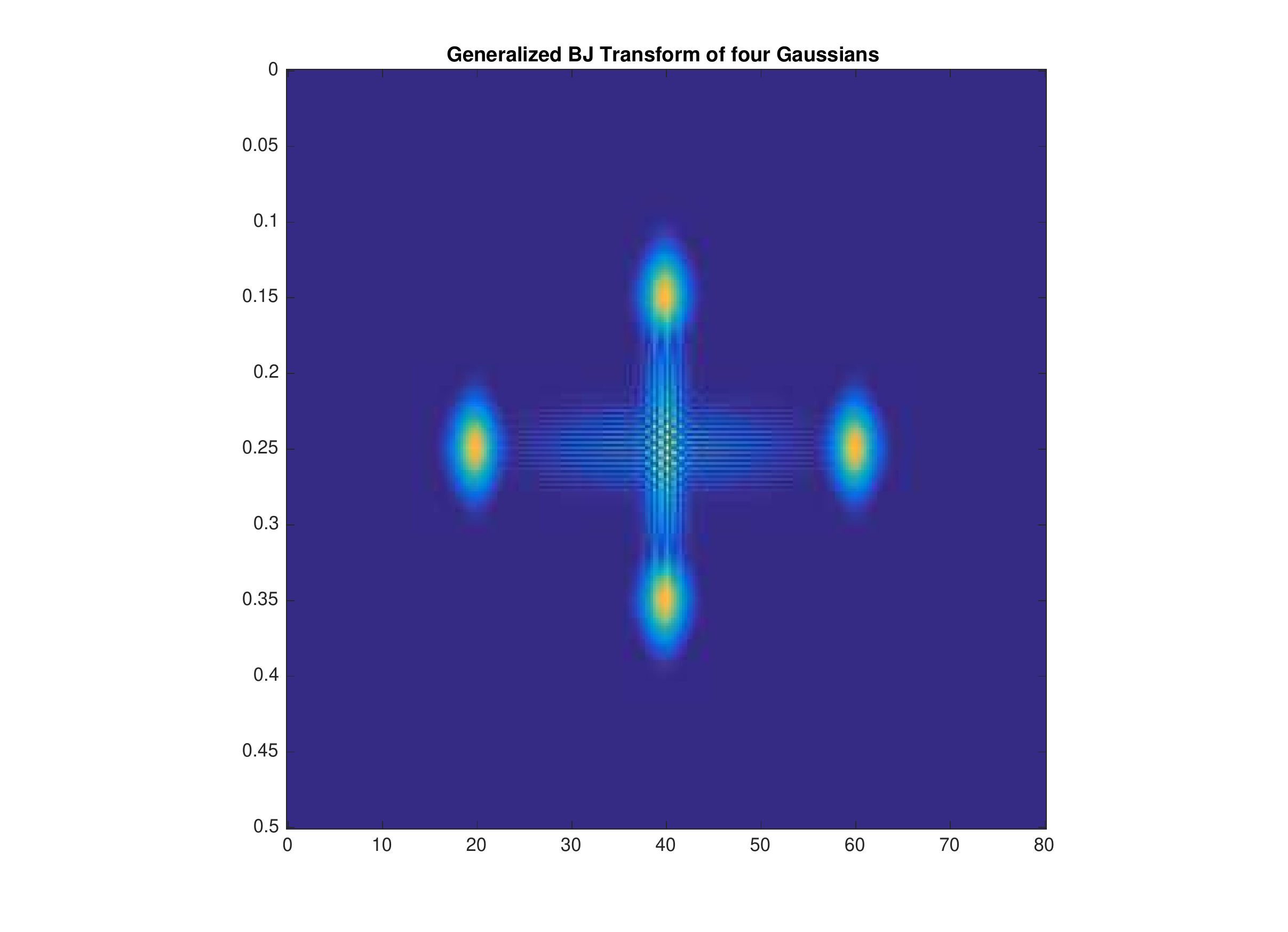}
\end{figure}

Notice that this new  class is a subset of the Cohen class,
containing the Wigner and the classical BJ  distribution. This subclass can be applied to signal processing since the mathematical explanations of their smoothing properties testifies the numerical evidences of dumping artefacts in many examples, and such reduction increases with $n$.
We survey  the different facets of this phenomenon,  referring mainly to the theoretical results in \cite{BJsplines2019}, highlighted  by new pictures in this note. \par
First, in Figures \ref{figWV}, \ref{figBJ} and \ref{figBJ3} we show the Wigner, BJ, BJ3 distribution of a rotation of the original  signal in \eqref{signalf}.
\begin{figure}
	\includegraphics[width=1.0\textwidth]{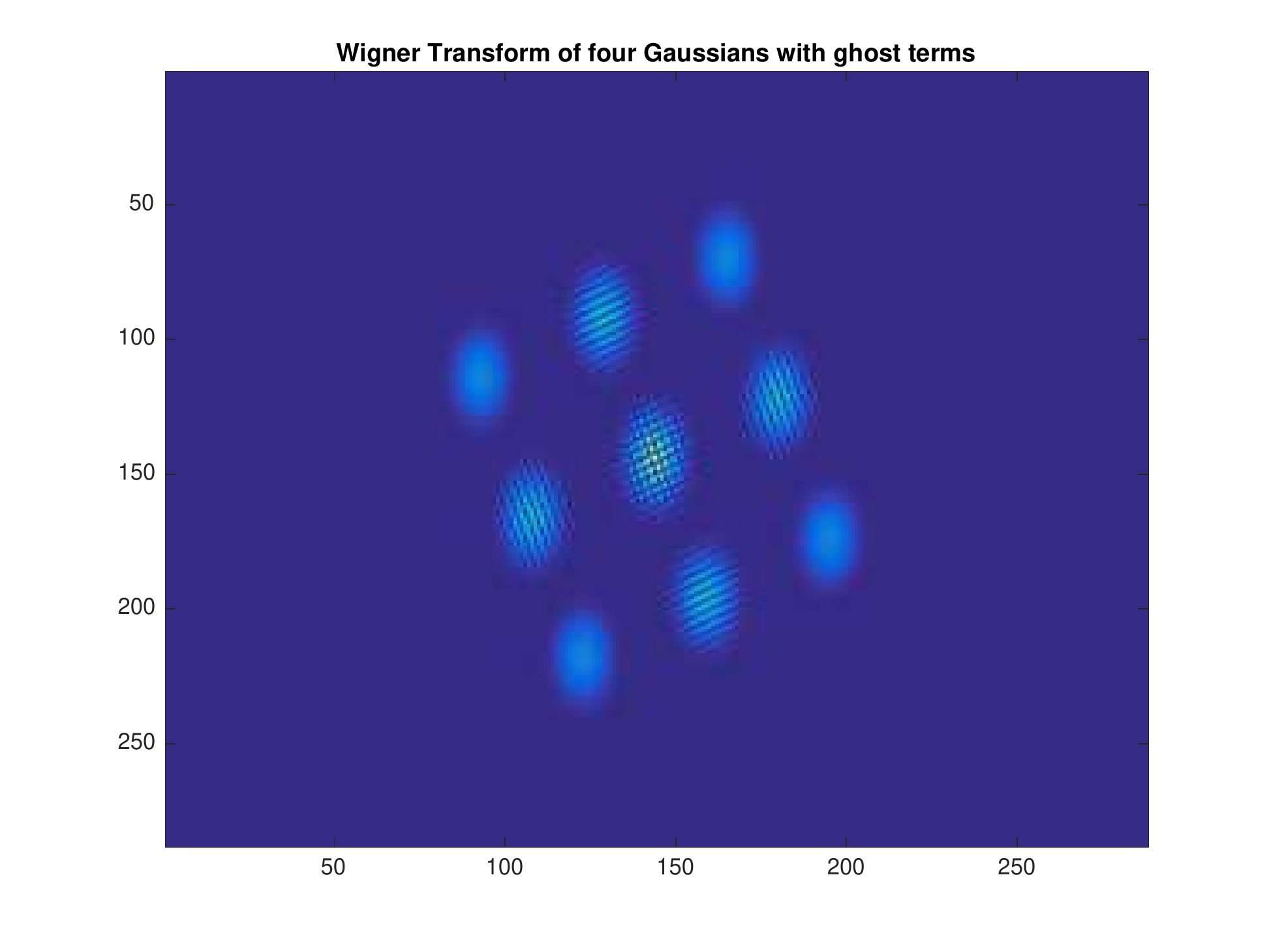}
	\caption{The Wigner Transform shows $9$ spots, only $4$ spots represent the real signal, the others are ghost terms}\label{figWV}
\end{figure}
\begin{figure}
	\includegraphics[width=1.0\textwidth]{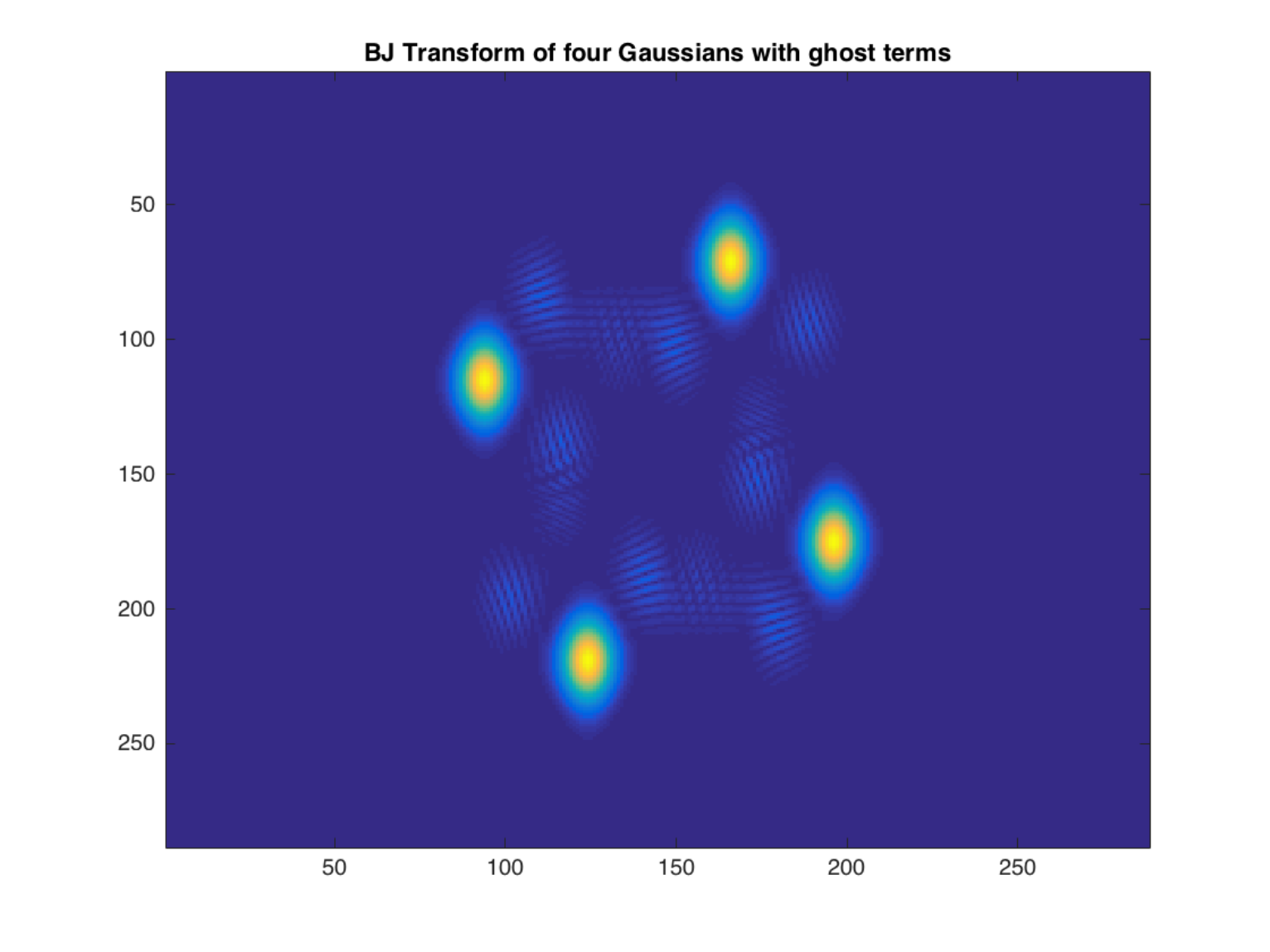}
		\caption{The BJ Transform of the previous signal gives a better information than the Wigner one. Here one sees $4$ spots that represent the time-frequency content of the signal though  few little shades appear as interferences}\label{figBJ}
\end{figure}
\begin{figure}
	\includegraphics[width=1.0\textwidth]{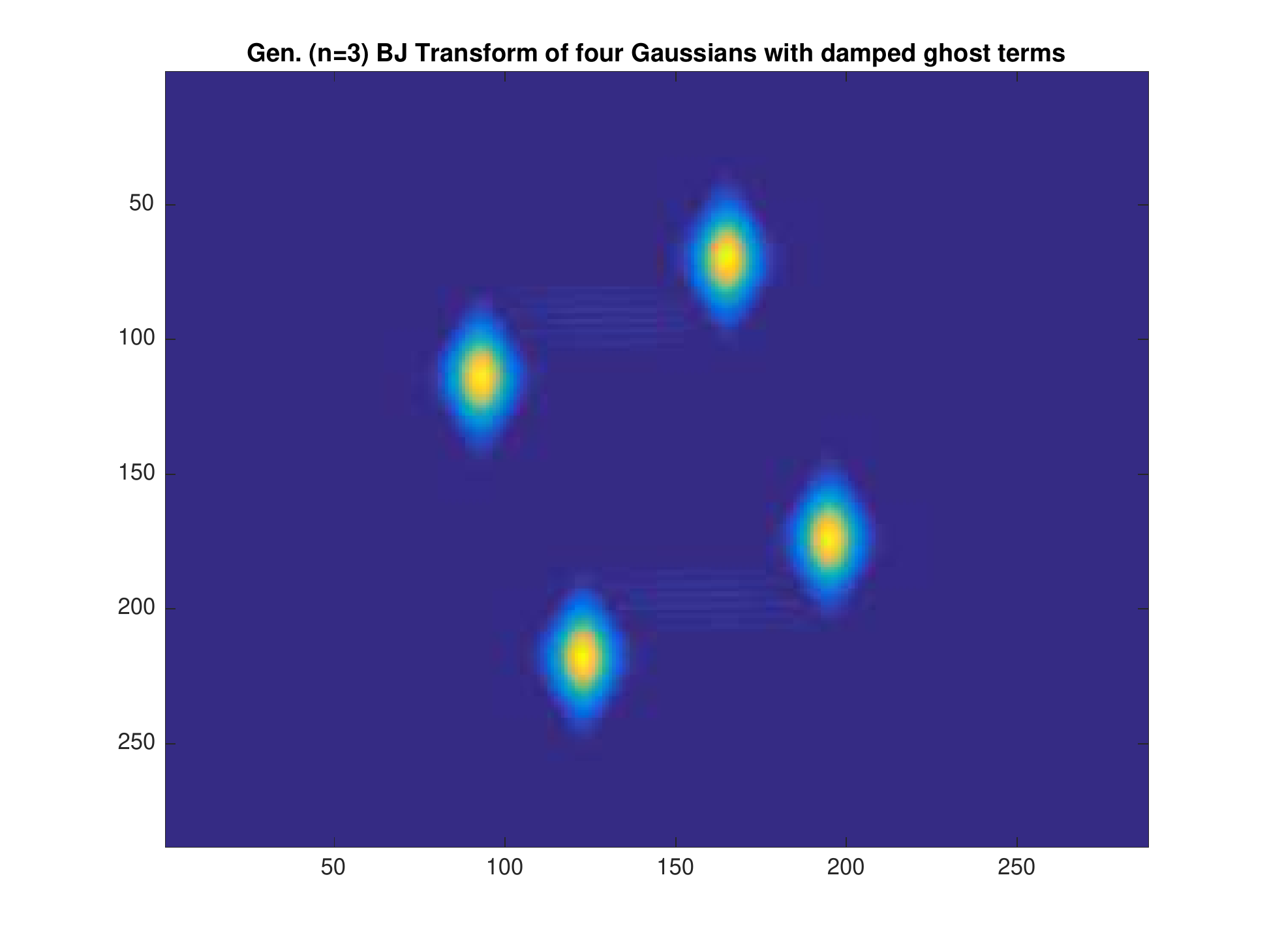}
		\caption{The Generalized BJ Transform (order $n=3$) of the same signal shows the best time-frequency content of it. The interferences are almost disappeared}\label{figBJ3}
\end{figure}
 Figure \ref{fig:GenBJ}  shows a comparison of the Wigner transform, the Born-Jordan transform and the fifth Born-Jordan transform of another sum of  four time-frequency shifts of Gaussian functions. It is clearly visible, that the amount of cross-term suppression increases by applying higher-order smoothing.\\
 
\begin{figure}[htbp]
	\begin{center}
		\hspace{-0.5cm}\includegraphics[width=1.1\textwidth]{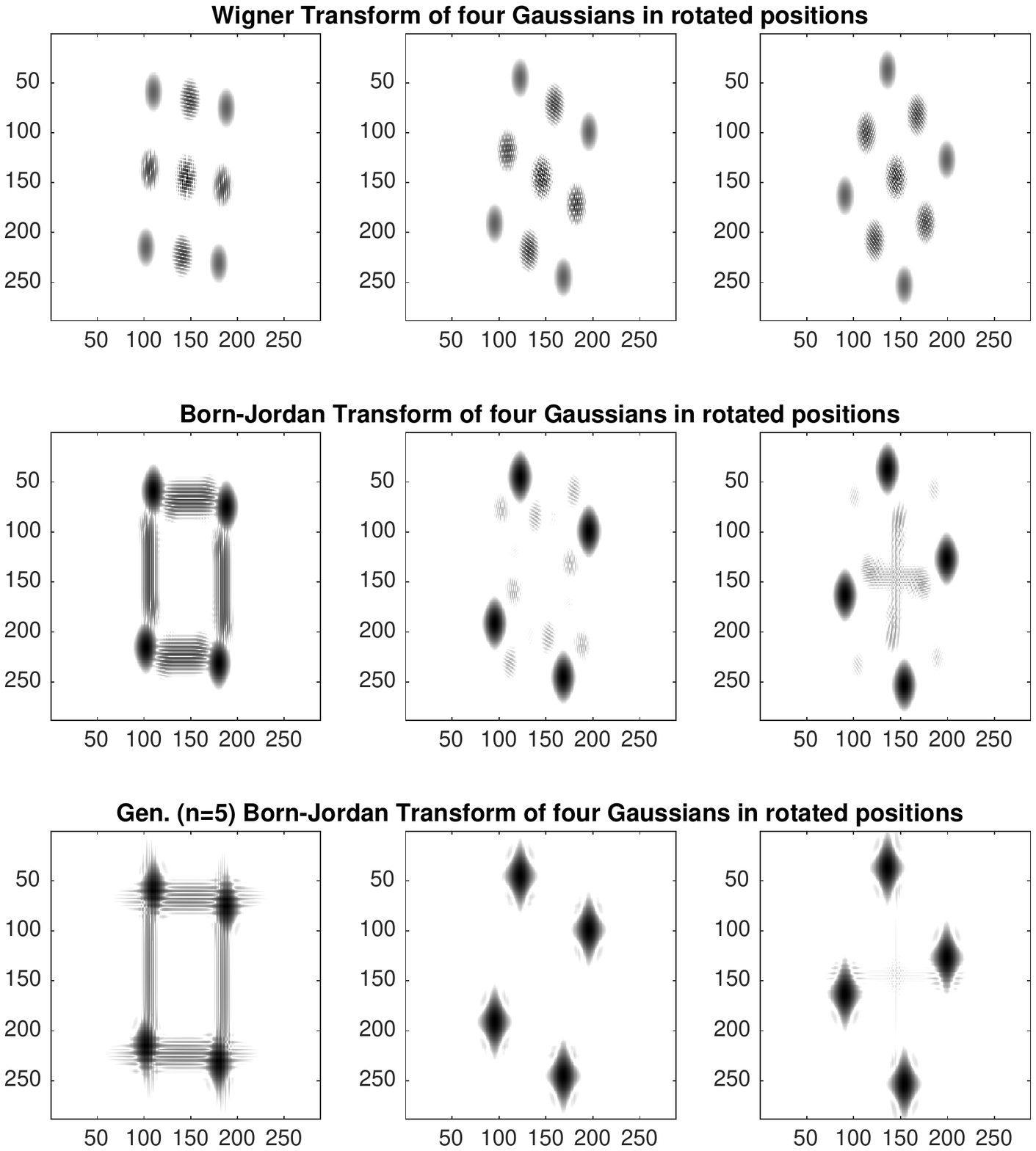}
		\caption{Four Gaussian functions in rotated positions:  Comparison of Wigner distribution, Born-Jordan and generalised Born-Jordan distribution of order $n=5$}
		\label{fig:GenBJ}
	\end{center}
\end{figure}
This suggests that such distributions could be successfully applied in signal processing.\par 
In the following sections we provide a rigorous  mathematical explanation of the 
regularity and smoothness properties of $Q^{n}$; the notion of Fourier Lebesgue wave-front set will play the central role in  showing the damping of interferences of $Q^{n}$ in comparison with the Wigner
	distribution. The use of  function spaces from time-frequency analysis, combined with   microlocal analysis techniques are the key tools of our proofs.

\subsection{Notation}

We use $x\omega=x\cdot\omega=x_{1}\omega_{1}+\ldots+x_{d}\omega_{d}$ for the
scalar product in $\mathbb{R}^{d}$. The brackets  $\langle\cdot,\cdot\rangle$ denote the inner
product in $L^{2}(\mathbb{R}^{d})$ or  the duality pairing between Schwartz
functions and temperate distributions (antilinear in the second argument).
For functions $f,g$, we write $f\lesssim g$ if $f(x)\leq C g(x)$ for every
$x$ and some constant $C>0$, and similarly for $\gtrsim$. The notation
$f\asymp g$ means $f\lesssim g$ and $f\gtrsim g$.
We write $\mathcal{C}^{\infty}_{c}(\mathbb{R}^{d})$ for the class of smooth
functions on $\mathbb{R}^{d}$ with compact support.
The notation  $\sigma$ stands for the standard symplectic form on the phase space
$\mathbb{R}^{2d}\equiv\mathbb{R}^{d}\times\mathbb{R}^{d}$; the phase space
variable is denoted $z=(x,\omega)$ and the dual variable by $\zeta=(\zeta
_{1},\zeta_{2})$. By definition $\sigma(z,\zeta)=Jz\cdot\zeta=\omega\cdot
\zeta_{1}-x\cdot\zeta_{2}$, where
\[
J=%
\begin{pmatrix}
0_{d\times d} & I_{d\times d}\\
-I_{d\times d} & 0_{d\times d}%
\end{pmatrix}
.
\]
The Fourier transform of a function $f$ in $\mathbb{R}^{d}$ is
\[
\mathcal{F} f(\omega)=\widehat{f}(\omega)= \int_{\mathbb{R}^{d}} e^{-2\pi i
	x\omega} f(x)\, dx,
\]
and the symplectic Fourier transform of a function $F$ in the phase space
$\mathbb{R}^{2d}$ is 
\[
\mathcal{F}_{\sigma}F(\zeta)=\int_{{\mathbb{R}^{2d}}} e^{-2\pi{i}\sigma
	(\zeta,z)} F(z)\, dz.
\]
The symplectic Fourier transform is an involution, i.e., $\mathcal{F}_{\sigma
}(\mathcal{F}_{\sigma}F)=F$. Moreover, $\mathcal{F}_{\sigma}F(\zeta)=
\mathcal{F} F(J \zeta)$.

Observe that $\Theta^{n}(J(\zeta_{1},\zeta_{2}))=\Theta^{n}(\zeta_{1}%
,\zeta_{2})$ so that
\begin{equation}
\label{cft}\mathcal{F}_{\sigma}(\Theta^{n})=\mathcal{F} (\Theta^{n}%
),\quad\forall n\in\mathbb{N}_{+}.
\end{equation}
\section{Time-frequency and Microlocal Analysis Methods}
In this section we recall the function spaces from time-frequency analysis and the wave-front set from microlocal analysis that play the key role in this study.\\
\noindent
\textbf{Modulation Spaces.}
Let us first recall another very popular 
time-frequency representation: \emph{the short-time Fourier transform (STFT)}.
Fix a Schwartz function $g\in\mathcal{S}(\mathbb{R}^{d})\setminus\{0\}$ (so-called \textit{window}), then the short-time Fourier transform of $f\in\ cS'(\rd)$ is
\begin{equation}
V_{g}f(x,\omega)=\int_{\mathbb{R}^{d}}f(y)\,{\overline{g(y-x)}}\,e^{-2\pi
	iy\omega}\,dy,\quad(x,\omega)\in{\mathbb{R}^{2d}}.\label{STFTdef}%
\end{equation}

 Consider $1\leq p,q\leq\infty$. The \textit{modulation
space} $M^{p,q}(\mathbb{R}^{d})$ consists of all tempered distributions
$f\in\mathcal{S}^{\prime}(\mathbb{R}^{d}) $ such that
\begin{equation}
\label{defmod}\|f\|_{M^{p,q}}:=\left( \int_{\mathbb{R}^{d}} \left(
\int_{\mathbb{R}^{d}}|V_{g}f(x,\omega) |^{p}\,
dx\right) ^{q/p}d\omega\right) ^{1/q}<\infty\,
\end{equation}
(with obvious modifications for $p=\infty$ or $q=\infty$).   We write $M^{p}(\mathbb{R}^{d})$ instead of $M^{p,p}(\mathbb{R}^{d})$. The modulation  spaces are Banach spaces
for any $1\leq p,q\leq\infty$, and  every non-zero $g\in\mathcal{S}%
(\mathbb{R}^{d})$ yields an equivalent norm in  \eqref{defmod}.

Modulation spaces were introduced in \cite{F1} and are now available in textbooks, see e.g., \cite{grochenig}. They include as special cases several function spaces arising in Harmonic Analysis. In particular for $p=q=2$ we have
\[
M^{2}(\mathbb{R}^{d})=L^{2}(\mathbb{R}^{d}),
\]
whereas $M^{1}(\mathbb{R}^{d})$ is the Feichtinger  algebra $S_{0}(\mathbb{R}^{d})$, cf. \cite{fei0,grochenig}.

In the notation $M^{p,q}$ the
exponent $p$ is a measure of decay at infinity (on average) in the scale of
spaces $\ell^{p}$, whereas the exponent $q$ is a measure of smoothness in the
scale $\mathcal{F} L^{q}$. 

Other instances of modulation spaces, also known as \textit{Wiener amalgam spaces}, are
obtained by exchanging the order of integration in \eqref{defmod}. Namely, for
$p,q\in\lbrack1,\infty)$, 
the modulation space $W(\mathcal{F}L^{p},L^{q})(\mathbb{R}^{d})$ is the subspace of tempered distributions $f\in\mathcal{S}%
^{\prime}(\mathbb{R}^{d})$ such that
\[
\Vert f\Vert_{W(\mathcal{F}L^{p},L^{q})(\mathbb{R}^{d})}:=\left(
\int_{\mathbb{R}^{d}}\left(  \int_{\mathbb{R}^{d}}|V_{g}f(x,\omega
)|^{p}\,d\omega\right)  ^{q/p}dx\right)  ^{1/q}<\infty\,
\]
(with obvious changes for $p=\infty$ or $q=\infty$). Using Parseval identity
in \eqref{STFTdef}, we infer the  fundamental identity of
time-frequency analysis\thinspace\
\[
V_{g}f(x,\omega)=e^{-2\pi ix\omega}V_{\hat{g}}\hat{f}(\omega,-x),
\]
hence
\[
|V_{g}f(x,\omega)|=|V_{\hat{g}}\hat{f}(\omega,-x)|=|\mathcal{F}(\hat
{f}\,T_{\omega}\overline{\hat{g}})(-x)|
\]
so that
\[
\Vert f\Vert_{{M}^{p,q}}=\left(  \int_{\mathbb{R}^{d}}\Vert\hat{f}\ T_{\omega
}\overline{\hat{g}}\Vert_{\mathcal{F}L^{p}}^{q}\ d\omega\right)  ^{1/q}%
=\Vert\hat{f}\Vert_{W(\mathcal{F}L^{p},L^{q})}.
\]
Hence Wiener amalgam spaces can be viewed as the image under
Fourier transform\thinspace\ of modulation spaces: $\mathcal{F}({M}%
^{p,q})=W(\mathcal{F}L^{p},L^{q})$.

We will frequently use the following product property of Wiener amalgam spaces
(\cite[Theorem 1 (v)]{feichtinger80}): For $1\leq p,q\leq\infty$,
\begin{equation}\label{product}
\textit{if $f\in W(\Fur L^1,L^\infty)$ and $g\in W(\Fur L^p,L^q)$ then $fg\in W(\Fur L^p,L^q)$}.
\end{equation}
Taking $p= 1, q=\infty$, we obtain that $W(\mathcal{F} L^{1},L^{\infty
})({\mathbb{R}^{2d}})$ is an algebra under point-wise multiplication.

\begin{proposition}
	\label{c1} Let $1\leq p,q\leq\infty$ and $A\in GL(d,\mathbb{R})$. Then, for
	every $f\in W(\mathcal{F}L^{p},L^{q})(\mathbb{R}^{d})$,
	\begin{equation}
	\label{dilAW0}\|f(A\,\cdot)\|_{W(\mathcal{F}L^{p},L^{q})}\leq C |\det
	A|^{(1/p-1/q-1)}(\det(I+A^{*} A))^{1/2}\|f\|_{W(\mathcal{F}L^{p},L^{q})}.
	\end{equation}
	In particular, for $A=\lambda I$, $\lambda>0$,
	\begin{equation}
	\label{dillambda}\|f(A\,\cdot)\|_{W(\mathcal{F}L^{p},L^{q})}\leq C
	\lambda^{d\left( \frac1p-\frac1q-1\right) }(\lambda^{2}+1)^{d/2}
	\|f\|_{W(\mathcal{F}L^{p},L^{q})}.
	\end{equation}
	
\end{proposition}
In the sequel we shall use the following fact \cite[Lemma 5.1]{ACHA2018}.
\begin{lemma}
	\label{lemma5.1}  Let $\chi\in C^{\infty}_{c}(\mathbb{R})$. Then the function
	$\chi(\zeta_{1} \zeta_{2})$ belongs to $W(\mathcal{F} L^{1},L^{\infty
	})({\mathbb{R}^{2d}})$.
\end{lemma}
\noindent
\textbf{Wave-front set for Fourier-Lebesgue spaces} 
For $s\in\mathbb{R}$ the  Sobolev space $H^{s}(\mathbb{R}^{d})$
is constituted by the distributions $f\in\mathcal{S}^{\prime}(\mathbb{R}^{d})$
such that
\begin{equation}
\label{normhs}\|f\|_{H^{s}}:=\| \widehat{f}(\omega) \langle\omega\rangle^{s}
\|_{L^{2}}<\infty.
\end{equation}
The  $H^{s}$ wave-front set allows to quantify the regularity of a
function/distribution in the Sobolev scale, at any given point and direction. This is achieved by microlocalizing the definition of the $H^{s}$ norm
in \eqref{normhs} as follows (cf.\ \cite[Chapter XIII]{hormander2}).

\begin{definition} \label{defhswave} Given a tempered distribution $f\in\mathcal{S}^{\prime}(\mathbb{R}^{d})$ its wave-front set $WF_{H^{s}} (f)\subset\mathbb{R}^{d}\times(\mathbb{R}%
^{d}\setminus\{0\})$ is  the set of points $({x}_{0},{\omega}_{0}%
)\in\mathbb{R}^{d}\times\mathbb{R}^{d}$, ${\omega}_{0}\not =0$, where the
following condition is \textit{not} satisfied: for some cut-off function
$\varphi\in C^{\infty}_{c}(\mathbb{R}^{d})$ with $\varphi({x}_{0})\not =0$ and
some open conic neighborhood of $\Gamma\subset\mathbb{R}^{d}\setminus\{0\}$ of
${\omega}_{0}$ we have
\[
\|\mathcal{F} [\varphi f](\omega) \langle\omega\rangle^{s}\|_{L^{2}(\Gamma
	)}<\infty.
\]
\end{definition}
More  generally one can start from the Fourier-Lebesgue spaces $\mathcal{F}
L^{q}_{s}(\mathbb{R}^{d})$, $s\in\mathbb{R}$, $1\leq q\leq\infty$, which is
the space of distributions $f\in\mathcal{S}^{\prime}(\mathbb{R}^{d})$ such
that the norm 
\begin{equation}
\label{eq4-0}\|f\|_{\mathcal{F} L^{q}_{s}(\mathbb{R}^{d})}=\|\widehat{f}%
(\omega) \langle\omega\rangle^{s}\|_{L^{q}(\mathbb{R}^{d})},
\end{equation}
 is finite. 
 
 Arguing exactly as in Definition \ref{defhswave} with the
space $L^{2}$ replaced by $L^{q}$, one can introduce  the
corresponding notion of wave-front set $WF_{\mathcal{F} L^{q}_{s}}(f)$.
\begin{definition}
Given $f\in\mathcal{S}^{\prime}(\mathbb{R}^{d})$
its wave-front set $WF_{\mathcal{F} L^{q}_{s}} (f)\subset\mathbb{R}^{d}%
\times(\mathbb{R}^{d}\setminus\{0\})$ is the set of points $({x}_{0},{\omega
}_{0})\in\mathbb{R}^{d}\times\mathbb{R}^{d}$, ${\omega}_{0}\not =0$, where the
following condition \textit{is not satisfied}: for some cut-off function
$\varphi$ (i.e., $\varphi$ is smooth and compactly supported on $\mathbb{R}%
^{d}$), with $\varphi({x}_{0})\not =0$, and some open conic neighbourhood
$\Gamma\subset\mathbb{R}^{d}\setminus\{0\}$ of ${\omega}_{0}$ it holds
\begin{equation}
\label{eq4}\|\mathcal{F} [\varphi f](\omega) \langle\omega\rangle^{s}%
\|_{L^{q}(\Gamma)}<\infty.
\end{equation}
\end{definition}
Observe that $WF_{\mathcal{F} L^{2}_{s}} (f)=WF_{H^{s}}(f)$ is the standard
$H^{s}$ wave-front set in Definition \ref{defhswave}.

For our purposes we recall some basic results about the action of constant coefficient linear partial differential operators on such wave-front set (cf. \cite{ptt1}).
Given the operator
\[
P=\sum_{|\alpha|\leq m} c_{\alpha}\partial^{\alpha},\quad c_{\alpha}%
\in\mathbb{C},
\]
 for $1\leq q\leq\infty$, $s\in\mathbb{R}$,
$f\in\mathcal{S}^{\prime}(\mathbb{R}^{d})$,
\[
WF_{\mathcal{F} L^{q}_{s}}(Pf) \subset WF_{\mathcal{F} L^{q}_{s+m}}(f).
\]
Consider now the inverse inclusion. We say that $\zeta\in\mathbb{R}^{d}$,
$\zeta\not =0$, is non characteristic for the operator $P$ if
\[
\sum_{|\alpha|=m} c_{\alpha}\zeta^{\alpha}\not =0.
\]
 The following result is a microlocal version of the classical regularity result of elliptic
operators (see \cite[Corollary 1 (2)]{ptt1}):
\begin{proposition}
	\label{pro3} Let $1\leq q\leq\infty$, $s\in\mathbb{R}$ and $f\in
	\mathcal{S}^{\prime}(\mathbb{R}^{d})$. Let $z\in\mathbb{R}^{d}$ and assume
	that $\zeta\in\mathbb{R}^{d}\setminus\{0\}$ is non characteristic for $P$.
	Then, if $(z,\zeta)\not \in WF_{\mathcal{F} L^{q}_{s}}(Pf)$, we have
	$(z,\zeta)\not \in WF_{\mathcal{F} L^{q}_{s+m}}(f)$.
\end{proposition}
\section{Time-frequency Analysis of the nth Born-Jordan kernel}
In this section we summarize the main results of the topic, obtained in the papers \cite{cgnb,ACHA2018,BJsplines2019}. 
The Born-Jordan kernel $\Theta^{1}$ in \eqref{sincxp} satisfies 
$$\Theta^{1}\in W(\mathcal{F} L^{1}, L^{\infty})({\mathbb{R}^{2d}}),$$
cf. \cite{ACHA2018}.

The previous property is true for any $\Theta^{n}$, $n\in\mathbb{N}_{+}$, as
shown below (see \cite{BJsplines2019}).
\begin{proposition}
	\label{pron}  For $n\in\mathbb{N}_{+}$, the function $\Theta^{n}$ in 
	\eqref{nCohenkerneln} belongs to the Wiener algebra $W(\mathcal{F}
	L^{1},L^{\infty})({\mathbb{R}^{2d}})$.
\end{proposition}
The properties of the kernels above yield the following result. 
\begin{theorem}
	\label{teo2-zero} Let $f\in\mathcal{S}^{\prime}(\mathbb{R}^{d})$ be a signal,
	with $Wf\in M^{p,q}({\mathbb{R}^{2d}})$ for some $1\leq p,q\leq\infty$. Then
	$Q^{n}f\in M^{p,q}({\mathbb{R}^{2d}})$, for every $n\in\mathbb{N}_{+}$.
\end{theorem}
\begin{proof}
	We need to  show that  $Q^{n} f$ is in $M^{p,q}({\mathbb{R}^{2d}})$. Taking the symplectic
	Fourier transform in \eqref{bj} this is equivalent to
	\[
	\Theta^{n} \mathcal{F}_{\sigma}(Wf)=\Theta^{n} Af\in W(\mathcal{F} L^{p}%
	,L^{q})
	\]
	where $\mathcal{F}_{\sigma}(Wf)=Af$ is called  the ambiguity function of $f$, see e.g., \cite{grochenig}. The claim is attained using the product property
	\eqref{product}: by Proposition \ref{pron}, the function $\Theta^{n}\in W(\mathcal{F} L^{1},L^{\infty})$ and by assumption $Wf\in M^{p,q}%
	({\mathbb{R}^{2d}})$ so that we infer $\mathcal{F}(Wf)\in W(\mathcal{F} L^{p},L^{q})$.
Finally,  $\mathcal{F}_{\sigma}(Wf)(\zeta)=\mathcal{F}(Wf)(J\zeta) \in
	W(\mathcal{F} L^{p},L^{q})$ by Proposition \ref{c1}.
\end{proof}

The previous statement holds in greater generality and can be rephrased for members in the Cohen class as follows.
\begin{theorem}
	\label{teo2-zero-Cohen} Let $f\in\mathcal{S}^{\prime}(\mathbb{R}^{d})$ be a signal,
	with $Wf\in M^{p,q}({\mathbb{R}^{2d}})$ for some $1\leq p,q\leq\infty$ and the Cohen kernel (cf. \eqref{Cohenkernel}) $\theta\in M^{1,\infty}(\rdd)$. Then the corresponding distribution $Qf$ is in
	$M^{p,q}({\mathbb{R}^{2d}})$.
\end{theorem}
\begin{proof} It is the consequence of the following  convolution relation for modulation spaces in \cite{CG02}:
	$$M^{p,q}({\mathbb{R}^{2d}})\ast
	M^{1,\infty}({\mathbb{R}^{2d}})\hookrightarrow M^{p,q}({\mathbb{R}^{2d}}),$$
	for any $1\leq p,q\leq \infty$.
\end{proof}
The chirp function $F(\zeta_{1},\zeta_{2})= e^{ 2\pi i \zeta_{1}
	\zeta_{2}}$ enjoys the following property, see \cite{ACHA2018,fei0}.

\begin{proposition}
	\label{pro1} The function $F(\zeta_{1},\zeta_{2})= e^{ 2\pi i \zeta_{1}
		\zeta_{2}}$ belongs to $W(\mathcal{F} L^{1},L^{\infty})({\mathbb{R}^{2d}})$.
\end{proposition}
By Proposition \ref{pro1} and by the dilation properties for Wiener amalgam
spaces in \eqref{dilAW0} we obtain
\begin{corollary}
	\label{cor1} For $\zeta=(\zeta_{1},\zeta_{2})$, consider the function
	$F_{J}(\zeta)=F( J \zeta)= e^{- 2\pi i \zeta_{1} \zeta_{2}}$.  Then $F_{J}\in
	W(\mathcal{F} L^{1},L^{\infty})({\mathbb{R}^{2d}})$.
\end{corollary}

\section{Smoothness of the Born-Jordan distribution of order $n$}
In this section we survey the results in \cite[Sec. 5]{BJsplines2019}, comparing the smoothness of the Born-Jordan
distribution of order $n$ with the Wigner distribution.

For  the following global result  we use the notation
\begin{equation}\label{defop}
\nabla_{x}\cdot\nabla_{\omega}:=\sum_{j=1}^{d} \frac{\partial^{2}}%
{\partial{x_{j}}\partial{\omega_{j}}}.
\end{equation}
\begin{theorem}
	\label{teo2} Let $f\in\mathcal{S}^{\prime}(\mathbb{R}^{d})$ be a signal, with
	$Wf\in M^{p,q}({\mathbb{R}^{2d}})$ for some $1\leq p,q\leq\infty$. Then, for any $n\in \bN_+$,
	\[
	Q^{n}f\in M^{p,q}({\mathbb{R}^{2d}})
	\]
	and 
	\begin{equation}
	\label{eq1}(\nabla_{x}\cdot\nabla_{\omega})^{n} Q^{n}f\in M^{p,q}%
	({\mathbb{R}^{2d}}).
	\end{equation}
	
\end{theorem}
\begin{proof}
	The claim   $Q^{n} f\in M^{p,q}({\mathbb{R}^{2d}})$ is proven in Theorem \ref{teo2-zero}.
	
	We now prove \eqref{eq1}. Taking the symplectic Fourier transform we see that
	it is sufficient to prove that
	\[
	(\zeta_{1}\zeta_{2})^{n}\, \mathrm{sinc}^{n}(\zeta_{1}\zeta_{2})
	\mathcal{F}_{\sigma}Wf=\frac{1}{\pi^{n}}\sin^{n}(\pi\zeta_{1}\zeta_{2})
	\mathcal{F}_{\sigma}Wf\in W(\mathcal{F} L^{p},L^{q}).
	\]
Using
	\begin{equation}
	\label{sin}\sin(\pi\zeta_{1}\zeta_{2})=\frac{e^{\pi i \zeta_{1}\zeta_{2}%
		}-e^{-\pi i \zeta_{1}\zeta_{2}}}{2i}
	\end{equation}
	and applying  Proposition \ref{pro1}, Corollary \ref{cor1} and Proposition \ref{c1}, with
	the scaling $\lambda=1/\sqrt{2}$, we get 
	$ \sin(\pi\zeta_{1}\zeta_{2})\in W(\mathcal{F} L^{1},L^{\infty})$.
	
	Hence, for $n=1$,
	\[
	\frac{1}{\pi}\sin(\pi\zeta_{1}\zeta_{2}) \mathcal{F}_{\sigma}Wf\in
	W(\mathcal{F} L^{p},L^{q})
	\]
	by the product property \eqref{product}. Assume now that, for a certain
	$n\in\mathbb{N}_{+}$,
	\[
	\frac{1}{\pi^{n}}\sin^{n}(\pi\zeta_{1}\zeta_{2}) \mathcal{F}_{\sigma}Wf\in
	W(\mathcal{F} L^{p},L^{q}).
	\]
	Then
	\[
	\frac{1}{\pi^{n+1}}\sin^{n+1}(\pi\zeta_{1}\zeta_{2}) \mathcal{F}_{\sigma
	}Wf=\underset{\in W(\mathcal{F} L^{1},L^{\infty})}{\underbrace{\frac1\pi
			\sin(\pi\zeta_{1}\zeta_{2})}}\cdot\underset{\in W(\mathcal{F} L^{p}%
		,L^{q})}{\underbrace{\frac{1}{\pi^{n}}\sin^{n}(\pi\zeta_{1}\zeta_{2})
			\mathcal{F}_{\sigma}Wf}} \in W(\mathcal{F} L^{p},L^{q}),
	\]
	by \eqref{sin} and the product property \eqref{product} again. By induction we
	attain the result.\par
\end{proof}

We can now provide the mathematical explanation of the $Q^{n}$'s smoothing effects: 
\begin{theorem}
	\label{mainteo}  Let $f\in\mathcal{S}^{\prime}(\mathbb{R}^{d})$ be a signal,
	with $Wf\in M^{\infty,q}({\mathbb{R}^{2d}})$ for some $1\leq q\leq\infty$. Let
	$({z},{\zeta})\in{\mathbb{R}^{2d}}\times{\mathbb{R}^{2d}}$, with ${\zeta
	}=({\zeta}_{1},{\zeta}_{2})$ satisfying ${\zeta}_{1}\cdot{\zeta}_{2}\not =0$.
	Then
	\[
	({z},{\zeta})\not \in WF_{\mathcal{F} L^{q}_{2n}}(Q^{n}f).
	\]	
\end{theorem}
\begin{proof}
Consider $n\in\mathbb{N}_{+}$. We will apply
	Proposition \ref{pro3} to the $2n$-th order operator $P^{n}$, where
	$P$ is  defined in \eqref{defop}. The non
	characteristic directions for $P^{n}$ are given by the vectors $\zeta
	=(\zeta_{1},\zeta_{2})\in\mathbb{R}^{d}\times\mathbb{R}^{d}$, satisfying
	$\zeta_{1}\cdot\zeta_{2}\not =0$. By \eqref{eq1} (with $p=\infty$) we have
	\[
	WF_{\mathcal{F} L^{q}}(P^{n} Q^{n} f)=\emptyset,
	\]
	because $\varphi F\in\mathcal{F} L^{q}$ if $\varphi\in C^{\infty}%
	_{c}({\mathbb{R}^{2d}})$ and $F\in M^{\infty,q}({\mathbb{R}^{2d}})$ (with
	$F=P^{n} Q^{n}f$). This implies
	\[
	(z,\zeta)\not \in WF_{\mathcal{F} L^{q}}(P^{n} Q^{n}f),\quad\forall
	(z,\zeta)\,\,\mbox{such \,that}\,\, \,\zeta=(\zeta_{1},\zeta_{2}%
	),\,\,\zeta_{1}\cdot\zeta_{2}\not =0.
	\]
	Since $\zeta$ is non characteristic for the operator $P^{n}$, by Proposition
	\ref{pro3} we infer
	\[
	(z,\zeta)\not \in WF_{\mathcal{F} L^{q}_{2n}}(Q^{n}f)
	\]
	for every $z\in{\mathbb{R}^{2d}}$.
\end{proof}

Roughly speaking, if the Wigner distribution $Wf$ has local regularity
$\mathcal{F} L^{q}$ and some control at infinity, then $Q^{n}f$ is smoother,
possessing \textbf{$s=2n$ additional derivatives}, at least in the directions
${\zeta}=({\zeta}_{1},{\zeta}_{2})$ satisfying ${\zeta}_{1}\cdot{\zeta}%
_{2}\not =0$. In dimension $d=1$ this condition reduces to ${\zeta}_{1}%
\not =0$ and ${\zeta}_{2}\not =0$. Hence this result explains the smoothing
phenomenon of such distributions, which involves all the directions except
those of the coordinates axes. 

\emph{This is the reason why the interferences of two components
which do not share the same time or frequency localization come out
substantially reduced}.

 Observe that for $n=1$ we recapture the damping
phenomenon of the classical Born-Jordan distribution (cf. \cite[Theorem 1.2]{ACHA2018}).

For signals in $L^{2}(\mathbb{R}^{d})$, the previous result can be rephrased in terms of the classical H\"ormander's wave-front set.

\begin{corollary}
	\label{cor} Let $f\in L^{2}(\mathbb{R}^{d})$, so that $Wf\in L^{2}%
	({\mathbb{R}^{2d}})$. Let $(z,\zeta)$ be as in the statement of Theorem
	\ref{mainteo}. Then $({z},{\zeta})\not \in WF_{H^{2n}}(Q^{n}f)$, that is the distribution $Q^{n}f$
	has regularity $H^{2n}$ at $z$ and in the direction $\zeta$.
\end{corollary}
\begin{proof}
We apply Theorem \ref{mainteo} with $q=2$. In fact,
	for $f\in L^{2}(\mathbb{R}^{d})$, Moyal's formula gives $Wf\in L^{2}%
	({\mathbb{R}^{2d}})=M^{2,2}(\mathbb{R}^{d})\subset M^{\infty,2}({\mathbb{R}%
		^{2d}})$, by inclusion relations for modulation spaces. Observe that the $\mathcal{F} L^{2}_{2n}$
	wave-front set coincides with the $H^{2n}$ wave-front set.
\end{proof}

What about smoothing effects  in the directions  ${\zeta}=({\zeta}_1,{\zeta}_2)$:  ${\zeta}_1\cdot{\zeta}_2=0$?

It seems that the smoothing effects could not occur in these directions, as  Fig. $1$  shows. The mathematical  explanation is below.

\begin{theorem}
	\label{teo3} Suppose that for some $1\leq p,q_{1},q_{2}\leq\infty$,
	$n\in\mathbb{N}_{+}$ and $C>0$, it occurs
	\begin{equation}
	\label{test}\|Q^{n}f\|_{M^{p,q_{1}}}\leq C\|W f\|_{M^{p,q_{2}}},
	\end{equation}
	for every $f\in\mathcal{S}(\mathbb{R}^{d})$. Then $q_{1}\geq q_{2}$.
\end{theorem}
\begin{proof} The main steps are as follows. 
	We will test the estimate \eqref{test} using rescaled Gaussian functions
	$f(x)=\varphi(\lambda x)$, with $\lambda>0$ large parameter. Restricting to a neighbourhood of $\zeta_{1}\cdot\zeta_{2}=0$, the
	constrain $q_{1}\geq q_{2}$ must be satisfied.
	
	An easy computation yields
	\begin{equation}
	\label{wignerdil}W(\varphi(\lambda\, \cdot))(x,\omega)=2^{d/2} \lambda^{-d}
	\varphi(\sqrt{2}\lambda\, x)\varphi(\sqrt{2}\lambda^{-1}\, \omega).
	\end{equation}
	For every $1\leq p,q\leq\infty$, the above formula gives
	\[
	\|W(\varphi(\lambda\, \cdot))\|_{M^{p,q}}=2^{d/2} \lambda^{-d}\| \varphi
	(\sqrt{2}\lambda\, \cdot) \|_{M^{p,q}}\|\varphi(\sqrt{2}\lambda^{-1}\, \cdot)
	\|_{M^{p,q}}.
	\]
	By the dilation properties of Gaussians (first proved in \cite[Lemma 1.8]{toft}, see also \cite[Lemma	3.2]{CNJFA2008}):
	\begin{equation}
	\label{eqa3}\|W(\varphi(\lambda\, \cdot))\|_{M^{p,q}}\asymp\lambda
	^{-2d+d/q+d/p}\quad\mathrm{as}\ \lambda\to+\infty.
	\end{equation}
	We now study the $M^{p,q}$-norm of the BJDn $Q^{n}(\varphi(\lambda\, \cdot))$.
  It will be estimated from below obtaining the same expansion as
	in \eqref{eqa3}. In detail,
	\[
	\|Q^{n}(\varphi(\lambda\, \cdot))\|_{M^{p,q}}=\|\mathcal{F}_{\sigma}%
	(\Theta^{n}) \ast W(\varphi(\lambda\, \cdot))\|_{M^{p,q}}.
	\]
	By taking the symplectic Fourier transform and using Lemma \ref{lemma5.1} and
	the product property \eqref{product} we have
	\begin{align*}
	\|\mathcal{F}_{\sigma}(\Theta^{n}) \ast W(\varphi(\lambda\, \cdot
	))\|_{M^{p,q}} & \asymp\|\Theta^{n} \mathcal{F}_{\sigma}[ W(\varphi(\lambda\,
	\cdot))]\|_{W(\mathcal{F} L^{p},L^{q})}\\
	& \gtrsim\|\Theta^{n}(\zeta_{1},\zeta_{2}) \chi(\zeta_{1}\zeta_{2}%
	)\mathcal{F}_{\sigma}[ W(\varphi(\lambda\, \cdot))]\|_{W(\mathcal{F}
		L^{p},L^{q})}%
	\end{align*}
	for any $\chi\in C^{\infty}_{c}(\mathbb{R})$ and $n\in\mathbb{N}_{+}$.
	Choosing the function  $\chi$ compactly supported in the interval $[-1/4,1/4]$ and $\chi\equiv1$ in
	the interval $[-1/8,1/8]$ (the latter condition will be used later), we write
	\[
	\chi(\zeta_{1} \zeta_{2})=\chi(\zeta_{1} \zeta_{2}) \Theta^{n}(\zeta_{1}%
	,\zeta_{2})\Theta^{-n}(\zeta_{1},\zeta_{2})\tilde{\chi}(\zeta_{1} \zeta_{2}),
	\]
	with $\tilde{\chi}\in C^{\infty}_{c}(\mathbb{R})$ supported in $[-1/2,1/2]$
	and $\tilde{\chi}=1$ on $[-1/4,1/4]$, therefore on the support of $\chi$.
	Since by Lemma \ref{lemma5.1} the function $\Theta^{-n}(\zeta_{1},\zeta
	_{2})\tilde{\chi}(\zeta_{1} \zeta_{2})$ belongs to $W(\mathcal{F}
	L^{1},L^{\infty})$, by the product property the last expression can be
	estimated from below as
	\[ \|\chi(\zeta_1\zeta_2)\|_{W(\mathcal{F} L^{p},L^{q})}
	\gtrsim\| \chi(\zeta_{1}\zeta_{2})\mathcal{F}_{\sigma}[ W(\varphi(\lambda\,
	\cdot))]\|_{W(\mathcal{F} L^{p},L^{q})}.
	\]
	Finally in the proof of  \cite[Theorem
	1.4]{ACHA2018} it was shown 
	\begin{equation}
	\label{e27}\| \chi(\zeta_{1}\zeta_{2})\mathcal{F}_{\sigma}[ W(\varphi
	(\lambda\, \cdot))]\|_{W(\mathcal{F} L^{p},L^{q})}\gtrsim\lambda
	^{-2d+d/p+d/q}\quad\mathrm{as}\ \lambda\to+\infty.
	\end{equation}
	Comparing \eqref{e27} with \eqref{eqa3} we obtain the desired conclusion.
\end{proof}

\section{Conclusion and Perspectives}

The generalized Born-Jordan distributions presented in these notes  produce an improvement in the  damping of unwanted artefacts of some signals as the one represented in Fig. 1. For other pictures of  signals where the smoothing effects are visible  we refer to the original paper \cite{BJsplines2019}.

Let us underline that the emergence of interferences  is a well-known drawback of any quadratic representation, the distributions BJDn are not immune to this phenomenon, as it can be seen in the following pictures, which are  short extracts from  real music signals. For comparison, we also used the Spectrogram, which is another member of the Cohen class  (see \cite{bogetal} and references therein), given by
$$|V_f f\phas|^2,\quad f\in\lrd.$$

\begin{figure}[htbp]
	\begin{center}
		\includegraphics[width=1.0\textwidth]{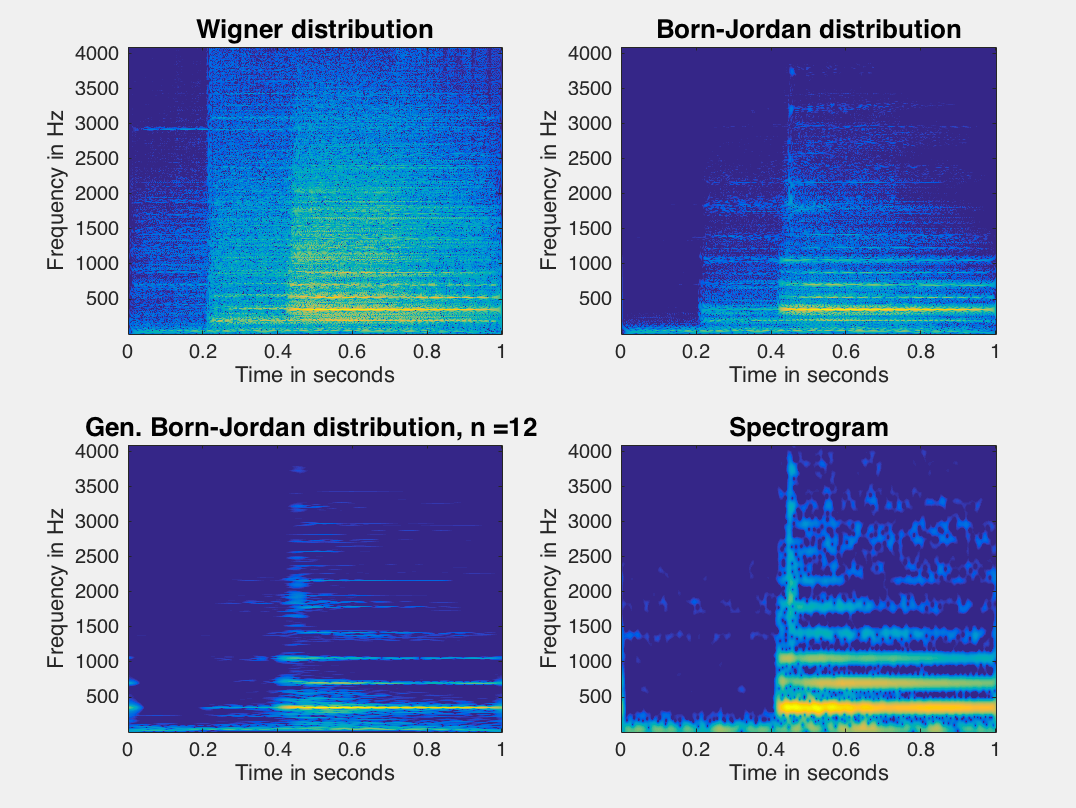}
		\caption{A short extract of a musical signal}
		\label{fig:M2}
	\end{center}
\end{figure}
\begin{figure}[htbp]
	\begin{center}
		\includegraphics[width=1.0\textwidth]{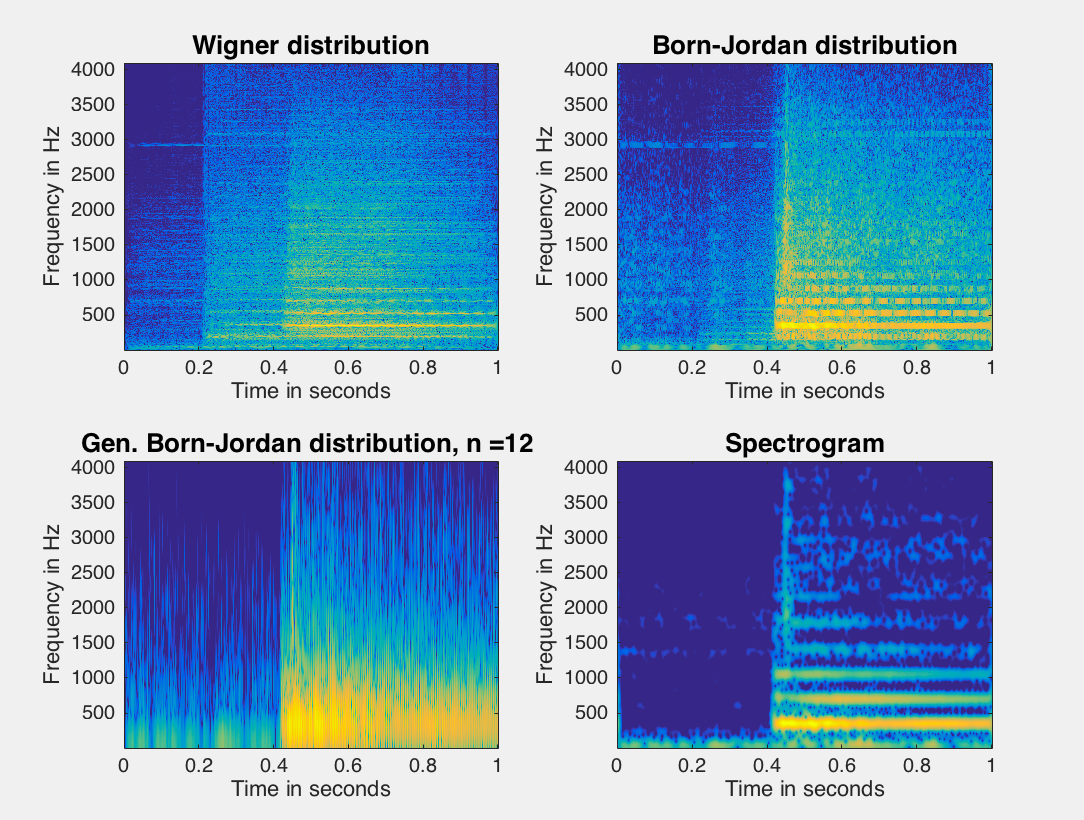}
		\caption{A short extract of a musical signal}
		\label{fig:M1}
	\end{center}
\end{figure}

Many other alternative time-frequency distributions  have been proposed for different practical  purposes, we refer the interested reader to the textbooks \cite{Cohen2,auger}.

If we consider  linear perturbations of the Wigner distribution, introduced and studied in \cite{bayer, bayerprocLuigi}, then  it turns out they do not provide effective damping of artefacts.  
In fact a negative answer concerning reduction of interferences is shown in \cite{EI}.

Although it is clear there is no time-frequency distribution which is  the best time-frequency representation  to analyse any kind of signal, it is an open problem to find the right distribution for a certain class of signals.  
For signals which are time-frequency shifts of Gaussians  the representations BJDs work quite well, as shown in the present note. This suggests the following question:

\emph{Concerning the members of the Cohen class, what are the best possible kernels for damping artefacts?}
 
\vspace{0.5truecm}

{\bf Technical notes.} The figures in these notes were produced using LTFAT (The Large Time-Frequency Analysis Toolbox), cf.~\cite{ltfat}.

\section*{Acknowledgments}
Monika Dörfler has been supported by the Vienna Science and Technology Fund (WWTF) through project MA14-018

\end{document}